\documentclass[11pt,a4paper,twoside,english]{article}

\usepackage{amsmath}
\usepackage{amssymb}
\usepackage{amsopn}
\usepackage{amsthm}
\usepackage{mathrsfs}
\usepackage{mathtools,xparse}
\usepackage{theoremref}
\usepackage{combelow}
\usepackage{pifont}
\usepackage{pdfpages}

\newtheorem{theorem}{Theorem}
\newtheorem{lemma}[theorem]{Lemma}
\newtheorem{proposition}[theorem]{Proposition}
\newtheorem{corollary}[theorem]{Corollary}

\newtheorem{remark}[theorem]{Remark}

\newcommand{\diffto}{\xrightarrow{\raisebox{-0.2 em}[0pt][0pt]{\smash{\ensuremath{\sim}}}}}
\newcommand{\R}{\mathbb{R}}
\newcommand{\N}{\mathbb{N}}
\newcommand{\norm}[1]{\|#1\|}
\DeclareMathOperator{\Lie}{\mathcal{L}}

\newcommand{\Addresses}{{
  \bigskip
  Ioan M\u{a}rcu\cb{t},\par\nopagebreak
 
\textsc{Radboud University Nijmegen, {\footnotesize 6500 GL Nijmegen, The Netherlands}}\par\nopagebreak
  \textit{E-mail address}: \texttt{i.marcut@math.ru.nl}
}}

\title{Poisson structures whose Poisson diffeomorphism group is not locally path-connected}
\author{Ioan M\u{a}rcu\cb{t}}

\begin{document}
\date{}
\maketitle

\begin{abstract}
We build examples of Poisson structure whose Poisson diffeomorphism group is not locally path-connected. 
\end{abstract}

\section*{Introduction}

Let $(M,\pi)$ be a Poisson manifold. The Poisson diffeomorphisms of $(M,\pi)$
\[\mathrm{Poiss}(M,\pi):=\left\{\phi:M\diffto M\, : \, \phi_*(\pi)=\pi\right\}\subset \mathrm{Diff}(M)\]
form a $C^1$-closed subgroup of the group of all diffeomorphisms. In analogy to Cartan's closed-subgroup theorem for finite-dimensional Lie groups, a natural question is whether $\mathrm{Poiss}(M,\pi)$ is, in some sense, a smooth subgroup of $\mathrm{Diff}(M)$. To simplify the discussion, we consider only the compact case; a convenient setting for dealing with the non-compact case was developed in \cite{Michor}. If $M$ is compact, then $\mathrm{Diff}(M)$ endowed with the $C^{\infty}$-topology has the structure of a Fr\'echet Lie group with Lie algebra the space of all vector fields $\mathfrak{X}(M)$ \cite{Hamilton,Milnor}. A suitable candidate for the Lie algebra of $\mathrm{Poiss}(M,\pi)$ is formed by the Poisson vector fields:
\[\mathfrak{poiss}(M,\pi):=\left\{ X\in \mathfrak{X}(M)\, : \, \Lie_X(\pi)=0\right\}\subset \mathfrak{X}(M).\]
Very little is known about whether $\mathrm{Poiss}(M,\pi)$ is a Lie subgroup of $\mathrm{Diff}(M)$, but the answer to this question must depend very much on the nature of the Poisson structure. In the symplectic case, it was proven by Weinstein that the symplectomorphism group is a Lie subgroup, and the local charts were obtained as a consequence of the his Lagrangian tubular neighborhood theorem applied to the diagonal in $M\times \overline{M}$ \cite{Weinstein}. For a regular Poisson manifold $(M,\pi)$, using a foliated version of Weinstein's argument, a Lie group structure can be constructed on the group $\mathrm{Fol}(M,\pi)$ consisting of Poisson diffeomorphisms that send each symplectic leaf to itself \cite{Rybicki} (see also \cite{Banyaga} for a more restrictive version of this result). In general, this subgroup is not closed (as an example, consider an invariant Poisson structure on the 3-torus with dense leaves), therefore the natural topology on this subgroup from \cite{Rybicki} is in general finer than the $C^{\infty}$-topology (compared to \cite{Banyaga}).

In this note, we construct Poisson structures whose Poisson diffeomorphism group is not locally path-connected, and so it is not a Lie group with the $C^{\infty}$-topology. Our first example is compactly supported in a 2-disk, and can be grafted using a local chart on any compact surface. By considering products with other Poisson manifolds, one obtains compact examples in higher dimensions. The second example is a regular Poisson structure given as a family of symplectic structures on a surface, with parameter space a second surface. Also in this case, by applying standard modifications, one can obtain compact regular examples in higher dimensions and codimensions. However, we do not know whether there exist co-rank one Poisson structures whose Poisson diffeomorphism group is not locally path-connected. Our examples are not analytic, it would be interesting to investigate whether analytic Poisson structures can have a Poisson diffeomorphism group which is not locally path-connected. Another interesting phenomenon that we observe in our examples is that any neighborhood of the identity intersects an uncountable number of path-components of the Poisson diffeomorphism group; it would be interesting to know whether examples without this feature exist.\\

\noindent \textbf{Acknowledgements.} I have learned about the question of whether the Poisson diffeomorphism group is locally path-connected from Rui Loja Fernandes, and I have talked with Jo\~ao Nuno Mestre about this problem and the specific construction presented in the paper. I would like to thank both of them for these useful discussions.

\section{The first example}

Let $\Sigma$ be a compact surface. Fix a closed disk $B\subset \Sigma$ contained in a coordinate chart $x=(x_1,x_2)$ which identifies $B$ with the closed disk of radius 1 around 0. We construct a bivector $\pi$ on $\Sigma$ whose support is contained in the interior of $B$ and consists of a disjoint union of disks accumulating at $0$. The centers of the disks are arranged on circles of decreasing radii around $0$. For each circle, we build a Poisson diffeomorphism which rotates the disks on that circle, and fixes all other disks and the exterior of $B$. We obtain a sequence of Poisson diffeomorphisms converging to the identity, such that none of its elements can be joined continuously via Poisson diffeomorphisms to the identity. The only cumbersome part of the construction is to fix all coefficients such that everything converges in the $C^{\infty}$-topology. 

For the analysis involved, we use the notation: for $a=(a_1,a_2)\in \N^2$, let
\[|a|:=a_1+a_2, \ \ \ \ D^a:=\frac{1}{a_1!}\frac{1}{a_2!}\partial_{x_1}^{a_1}\partial_{x_2}^{a_2}.\]
Define the $C^k$-norm of a function $f\in C^{\infty}(B)$ as:
\[\norm{f}_k:=\sup\left\{\big|D^af(x)\big|\  : \  x\in B, \ a\in \N^2, \ |a|\leq k\right\}.\]
The $C^k$-norms extend to vector-valued maps or more general tensors as the maximum of the $C^k$-norms of the components.

We fix a smooth function $\chi:\R\to [0,1]$, which will be used in the construction of the Poisson structure and of the diffeomorphisms, with the following properties: 
\[\chi(t)=\begin{cases} 0, & 1\leq |t| \\
>0,  & 1/2 < |t|<1\\
1,  & 0 \leq |t|\leq 1/2\\
\end{cases}.\]
Since we work in dimension two, any smooth bivector field $f(x_1,x_2)\partial_{x_1}\wedge \partial_{x_2}$ is a Poisson structure. For $p\in \R^2$ and $\delta>0$, define
\[\pi_{p,\delta}(x):=\chi\left(\frac{|x-p|}{\delta}\right)\partial_{x_1}\wedge\partial_{x_2}.\]
Note that $\mathrm{supp}(\pi_{p,\delta})=\overline{B}_\delta(p)$. We will construct the Poisson structure as an infinite linear combination of bivectors of the form $\pi_{p,\delta}$. To insure convergence, we will use the following: 

\begin{lemma}\label{lemma_bivector}
For each $k\geq 0$, there is $C_k>0$ depending on $\chi$, such that 
\[\norm{\pi_{p,\delta}}_{k}\leq C_k \delta^{-k}, \ \ \forall \ p\in \R^2,\ \ 0<\delta.\]
\end{lemma}
\begin{proof}
We claim that, for any $a\in \N^2$ with $|a|=k$, there is $C_k>0$ such that 
\begin{equation}\label{ineq_length}
    \left|D^a(|x|)\right|\leq C_k|x|^{1-k}.
\end{equation}
For this, one shows inductively that there is are homogeneous polynomials of degree $k$, denoted  $p_a(x_1,x_2)$, such that 
\[D^a(|x|)=|x|^{1-2k}p_a(x_1,x_2),
\]
and then one uses the obvious estimate:
\[|p_a(x_1,x_2)|\leq C_k|x|^k.\]

For $k=0$, the statement is obvious. Since the $C^k$-norms are translation-invariant, it suffices to prove the inequalities for $\pi_{0,\delta}$. Fix $a\in \N^2$ with $|a|=k>0$. By the general chain rule, we have: 
\begin{equation*}
    D^a\left(\chi(|x|/\delta)\right)=\sum_{a^1+\ldots+a^i=a}\chi^{(i)}(|x|/\delta)\delta^{-i} D^{a^1}(|x|)\ldots D^{a^i}(|x|),
\end{equation*}
where the sum runs over all decomposition of $a$ into elements $a^j\in \N^2$, with $|a^j|\geq 1$. 
Applying \eqref{ineq_length}, we obtain
\begin{equation}\label{chain_rule}
    \left|D^a\left(\chi(|x|/\delta)\right)\right|\leq C_a\sum_{1\leq i\leq k}|\chi^{(i)}(|x|/\delta)|\,\delta^{-i} |x|^{i-k},
\end{equation}
for some constant $C_a>0$. If $|x|<\delta/2$ we have that $\chi^{(i)}(|x|/\delta)=0$, so all elements vanish. For $\delta/2\leq |x|$, \eqref{chain_rule} gives the estimate from the statement with $C_k$ depending on $k$ and on the norm $\norm{\chi}_k$.
\end{proof}

For $n\geq 4$, denote the corners of the regular $2^n$-gon lying on the circle of radius $1/n$ around the origin as follows: 
\[p(n,s):=\frac{1}{n}\exp\left(\frac{2 \pi s }{2^{n}}i\right), \ \ \ 1\leq s\leq 2^n,\]
where we identify $\R^2=\mathbb{C}$, and let \[\delta_n:=\frac{1}{n 2^n}.\] 
The numbers $\delta_n$ were chosen such that the closed disks $\overline{B}_{\delta_n}(p(n,s))$, for distinct $s$, do not intersect. This follows, because:
\[|p(n,s+1)-p(n,s)|=\frac{2}{n}\sin\left(\frac{\pi}{2^n}\right)>\frac{2}{n2^n}=2\delta_n.\]
Consider also the closed annuli $E_n\subset F_n$ given by:
\begin{align*}
    E_n&:=\left\{x\in \R^2\ :\ \frac{1}{n}-\frac{1}{4n^2}\leq |x|\leq \frac{1}{n}+\frac{1}{4n^2}\right\}\\
    F_n&:=\left\{x\in \R^2\ :\ \frac{1}{n}-\frac{1}{2n^2}\leq |x|\leq \frac{1}{n}+\frac{1}{2n^2}\right\}.
\end{align*}
The following is straightforward:
\begin{lemma}\label{lemma_annuli}
For all $n,m\geq 4$, with $n\neq m$, and all $1\leq s\leq 2^n$, we have that:
\[\overline{B}_{\delta_n}(p(n,s))\subset E_n\ \ \textrm{and}\ \ E_n\cap F_m=\emptyset.\] 
\end{lemma}

For $n\geq 4$, define

\[\pi_n:=\frac{1}{n!}\sum_{s=1}^{2^n}\pi_{p(n,s),\delta_{n}}.\]
Using Lemma \eqref{lemma_bivector}, and that the supports of the bivectors involved in the sum are disjoint, we obtain for each $k\geq 0$ a constant $C_k>0$, such that  
\[\norm{\pi_n}_k\leq C_k\frac{n^k2^{nk}}{n!}.\]
This implies that the series 
$\sum_n \pi_n$ 
converges uniformly in all $C^k$-norms. Consider the Poisson structure $\pi$ on $\Sigma$ given by
\[\pi|_B=\sum_{n=4}^{\infty}\pi_n\ \ \ \textrm{and}\ \ \ \pi|_{\Sigma\backslash B}=0.\]
The support of $\pi$ is the compact set
\[K=\{0\}\cup \bigcup_{4\leq n}\bigcup_{1\leq s\leq 2^n}\overline{B}_{\delta_{n}}(p(n,s)).\]

Next, define $\phi_n\in\mathrm{Diff}(\Sigma)$, $n\geq 4$, such that $\phi_n|_{\Sigma\backslash B}=\mathrm{id}$ and for $x\in B$
\[\phi_n(x):=x\cdot \exp\left(\frac{2\pi}{2^n}i\, \chi\big(4n(n|x|-1)\big)\right),\]
where we identify $\R^2=\mathbb{C}$. 

The following implies that $\mathrm{Poiss}(\Sigma,\pi)$ is not locally path-connected. 
\begin{proposition}\label{prop}
The diffeomorphisms $\phi_n$ satisfy the following:
\begin{itemize}
    \item[(a)] $\phi_n$ preserves $E_n$ and $\phi_n|_{E_n}$ is the rotation through an angle of $\frac{2\pi}{2^{n}}$.
    \item[(b)] $\phi_n$ preserves $F_n$ and 
    $\phi_n$ is supported in $F_n$.
    \item[(c)]
    $\phi_n\in\mathrm{Poiss}(\Sigma,\pi)$.
    \item[(d)] There exist no continuous family
\[\phi_n^t\in \mathrm{Poiss}(\Sigma,\pi), \ \ 0\leq t\leq 1,\]
such that $\phi_n^0=\mathrm{id}$ and $\phi_n^1=\phi_n$.
    \item[(e)] The sequence $\phi_n$ converges in all $C^k$-norms to the identity map: \[\lim_{n\to\infty}\phi_n=\mathrm{id}.\]
\end{itemize}
\end{proposition}
\begin{proof}
The properties of $\chi$ imply immediately that $\phi_n$ satisfies (a) and (b). Note that $\pi_n$ is invariant under the rotation through an angle of $\frac{2\pi}{2^n}$. By Lemma \ref{lemma_annuli}, $\mathrm{supp}(\pi_n)\subset E_n$, and so, by (a), $(\phi_n)_*(\pi_n)=\pi_n$. On the other hand, for $n\neq m$, also by Lemma \ref{lemma_annuli} we have that $\mathrm{supp}({\pi}_m)\cap F_n=\emptyset$; thus $(\phi_n)_*(\pi_m)=\pi_m$.
These properties imply item (c). For item (d), note that, for such a family $\phi_n^t$, the rank of $\pi$ along the curve $\phi_n^t(p(n,1))$ would have be constant $2$, and therefore this curve cannot go out of $B_{\delta_n}(p(n,1))$ and reach $p(n,2)=\phi_n(p(n,1))$. 
Item (e) is implied by the following lemma.
\end{proof}
 
\begin{lemma}\label{lemma_inequality}
For each $k\geq 0$, there is $C_k>0$ depending on $\chi$, such that 
\[\norm{\phi_{n}-\mathrm{id}}_{k}\leq C_k \frac{n^{2k}}{2^n}.\]
\end{lemma}
\begin{proof}
Denote $f_n(x):=\frac{2\pi}{2^n}i\, \chi\big(4n(n|x|-1)\big)$. First, we prove the following:
\begin{equation}\label{yae}
\norm{f_n}_k\leq C_k\frac{n^{2k}}{2^n}.
\end{equation}
For $k=0$, this is obvious. By the same steps as in the proof of Lemma \ref{lemma_bivector}, we obtain for $a\in \N^2$, with $|a|=k>0$, a constant $C_a>0$ such that 
\begin{equation}\label{chain_rule_2}
    \left|D^a\left(\chi\big(4n(n|x|-1)\big)\right)\right|\leq C_a\sum_{1\leq i\leq k}|\chi^{(i)}\big(4n(n|x|-1)\big)|\, n^{2i} |x|^{i-k}.
\end{equation}
For $|x|<\frac{1}{2n}$, we have that 
$4n(n|x|-1)<-1$, and so $\chi^{(i)}\big(4n(n|x|-1)\big)=0$. For $|x|\geq \frac{1}{2n}$, \eqref{chain_rule_2} gives \eqref{yae} with $C_k$ depending on $k$ and $\norm{\chi}_k$.

Next, we prove the inequality:
\begin{equation}\label{yai}
\norm{e^{f_n}-1}_k\leq C_k\frac{n^{2k}}{2^n}.
\end{equation}
Using that $|e^{f_n(x)}|=1$, we obtain the inequality for $k=0$: 
\[\big|e^{f_n(x)}-1\big|=\big|\int_0^1f_n(x)e^{s f_n(x)}d s\big|\leq |f_n(x)|\leq \frac{2\pi}{2^n}.\]
For $|a|=k>0$, \eqref{yai} follows from the general chain rule and \eqref{yae}:
\[
\big|D^a(e^{f_n(x)}-1)\big|\leq \big|\sum_{a^1+\ldots+a^i}e^{f_{n}(x)}D^{a^1}(f_n(x))\ldots D^{a^i}(f_n(x))\big|\leq C_k\frac{n^{2k}}{2^n}.\]

Finally, the inequality from the statement follows from \eqref{yai}, by writing 
\[D^a(\phi_n(x)-x)=\sum_{a^1+a^2=a}D^{a^1}(x)D^{a^2}\big(e^{f_n(x)}-1\big),\]
and using that because $e^{f_n(x)}-1$ has support in the disk $B_1(0)$, on which $|D^{a^1}(x)|\leq 1$.
\end{proof}

\begin{corollary}\label{coro}
Every neighborhood of the identity in $\mathrm{Poiss}(\Sigma,\pi)$ intersects an uncountable number of path-connected components of $\mathrm{Poiss}(\Sigma,\pi)$. 
\end{corollary}
\begin{proof}
Let $\mathcal{U}$ be a neighborhood of the identity in $\mathrm{Poiss}(\Sigma,\pi)$. Then there is $\epsilon>0$ and $k\geq 0$ such that any $\phi\in \mathrm{Poiss}(\Sigma,\pi)$ with $\mathrm{supp}(\phi)\subset B$ and 
$\norm{\phi-\mathrm{id}}_k<\epsilon$ belongs to $\mathcal{U}$. With the notation from Lemma \ref{lemma_inequality}, let $n_{\epsilon}\geq 4$ be such that 
\begin{equation}\label{eq:sarmale}
 C_k\sum_{n\geq n_{\epsilon}}\frac{n^{2k}}{2^n}\leq \frac{\epsilon}{2}.
\end{equation}
Consider any sequence $\overline{u}=\{u_n\}_{n\geq n_{\epsilon}}$, with $u_n\in \{0,1\}$. Since the supports of the diffeomorphisms $\phi_n$ are disjoint, using Lemma \ref{lemma_inequality}, for any $m\geq n\geq n_{\epsilon}$, and any $l\geq 0$ the following holds:
\begin{equation}\label{eq:sarmale2}
\norm{\phi^{u_{m}}_{m}\circ \ldots\circ 
\phi^{u_{n}}_{n}-\mathrm{id}
}_l=\norm{(\phi^{u_{m}}_{m}-\mathrm{id})+\ldots +(\phi^{u_{n}}_{n}-\mathrm{id})
}_l\leq C_l\sum_{i\geq n}\frac{i^{2l}}{2^i}.
\end{equation}
This shows that the sequence $\{\phi^{u_{n}}_{n}\circ \ldots\circ
\phi^{u_{n_{\epsilon}}}_{n_{\epsilon}}\}_{n\geq n_{\epsilon}}$ converges to a Poisson diffeomorphism $\phi^{\overline{u}}$. By \eqref{eq:sarmale} and \eqref{eq:sarmale2}, we have that $\phi^{\overline{u}}\in \mathcal{U}$. Consider two different sequences $\overline{u}\neq \overline{v}$, say $u_n=1$ and $v_n=0$, for some $n$. Then $\phi^{\overline{u}}$ rotates the disks on the circle of radius $1/n$ and $\phi^{\overline{v}}$ does not. By the argument from the proof of item (d) of Proposition \ref{prop}, this implies that $\phi^{\overline{u}}$ and $\phi^{\overline{v}}$ belong to distinct path-components of $\mathrm{Poiss}(\Sigma,\pi)$.
\end{proof}

\section{The second example}

Consider a connected, oriented surface $\Lambda$, and let $\nu$ be a volume form on $\Lambda$. Consider a second compact manifold $\Sigma$ and a positive smooth function $f:\Sigma\to (0,\infty)$. We endow the product $\Sigma\times \Lambda$ with the regular Poisson structure $\pi_{\nu, f}$ whose symplectic leaves are 
\[\big(\{p\}\times \Lambda, f(p)\,\nu\big), \ \ p\in \Sigma.\] 
Note that the group of diffeomorphisms of $\Sigma$ which preserve $f$, denoted by
\[\mathrm{Diff}(\Sigma,f):=\left\{\phi:\Sigma\diffto \Sigma\, : \, f\circ \phi=\phi \right\}\subset \mathrm{Diff}(\Sigma),\]
can be embedded into the group of Poisson diffeomorphisms of $\pi_{\nu, f}$:
\[i:\mathrm{Diff}(\Sigma,f)\hookrightarrow\mathrm{Poiss}(\Sigma\times \Lambda,\pi_{\nu,f}), \ \ \ \ \phi\mapsto \phi\times\mathrm{id}.\]
Moreover, $i$ has a continuous right inverse, which is a group homomorphism
\[r:\mathrm{Poiss}(\Sigma\times \Lambda,\pi_{\nu, f})\to \mathrm{Diff}(\Sigma,f),\]
where $r(\Phi)(p)=q$ if and only if $\Phi$ sends the leaf $\{p\}\times \Lambda$ to the leaf $\{q\}\times \Lambda$; that $r(\Phi)\in \mathrm{Diff}(\Sigma,f)$, follows from the fact that $\Phi$ preserves the volume of the symplectic leaves, and so $f(p)=f(q)$. Thus, if we build a function $f$ such that $\mathrm{Diff}(\Sigma,f)$ is not locally path-connected, then also $\mathrm{Poiss}(\Sigma\times \Lambda,\pi_{\nu, f})$ is not locally path connected.  

\begin{remark}\rm
Let us remark that $\mathrm{Poiss}(\Sigma\times \Lambda,\pi_{\nu, f})$ is locally path-connected precisely when $\mathrm{Diff}(\Sigma,f)$ has this property. First note that the Poisson diffeomorphism group is the semi-direct product:
\[\mathrm{Poiss}(\Sigma\times \Lambda,\pi_{\nu, f})=\mathrm{Fol}(\Sigma\times \Lambda,\pi_{\nu, f})\rtimes i\big(\mathrm{Diff}(\Sigma,f)\big),\]
where $\mathrm{Fol}(\Sigma\times \Lambda,\pi_{\nu, f})$ consists of Poisson diffeomorphisms that send each leaf to itself. This group is a Lie group (hence, locally path-connected) when endowed with the $C^{\infty}$-topology; this follows using \cite{Rybicki}, and that the foliation is simple.
\end{remark}

To obtain an explicit example, we will rely on the work we have done so far. As in the previous section, let $\Sigma$ be a closed surface, and let $B\subset \Sigma$ a closed disk diffeomorphic to $\overline{B}_1(0)$, and write the Poisson structure constructed there as $\pi|_B=u(x_1,x_2)\partial_{x_1}\wedge\partial_{x_2}$. Extend $u$ to $\Sigma$ by $0$ outside of $B$, and let $f=u+1$. It is easy to see that satisfy $\phi_n\in \mathrm{Diff}(\Sigma,f)$. Also, note that $\phi_n$ permutes effectively some connected components of the open set $\{f>1\}$, and therefore it cannot be connected to the identity through a continuous path in $\mathrm{Diff}(\Sigma,f)$. Since $\lim_{n\to\infty}\phi_n=\mathrm{id}$, we obtain that $\mathrm{Diff}(\Sigma,f)$ is not locally path connected. Clearly, also a version of Corollary \ref{coro} holds for $\mathrm{Diff}(\Sigma,f)$. We conclude that any neighborhood of the identity in $\mathrm{Poiss}(\Sigma\times \Lambda,\pi_{\nu, f})$ intersects an uncountable number of path-connected components.

\Addresses
\end{document}